\documentclass[10pt]{amsart}
\usepackage{enumerate,amsmath,amssymb,latexsym,
amsfonts, amsthm, amscd, MnSymbol}

\theoremstyle{plain}

\newtheorem{theorem}{Theorem}

\numberwithin{equation}{section}

\begin{document}

\title {On Weierstrass $\wp$ in signature 3}

\date{}

\author[P.L. Robinson]{P.L. Robinson}

\address{Department of Mathematics \\ University of Florida \\ Gainesville FL 32611  USA }

\email[]{paulr@ufl.edu}

 \keywords{}

\begin{abstract}
 In his work on the Ramanujan theory of elliptic functions to alternative bases, Shen constructed two different elliptic functions in signature three; we determine the precise relationship between them, by precisely relating their coperiodic Weierstrass $\wp$ functions. 
\end{abstract}

\maketitle

\medbreak

The Weierstrass $\wp$ functions of our title are two in number. They were found by Li-Chien Shen in his contributions to the Ramanujan theory of elliptic functions to alternative bases, specifically in signature three. 

\medbreak 

In [2004] Shen fashioned an elliptic function in signature three by modifying a construction of the classical Jacobian elliptic functions: in the classical theory, the Jacobian functions may be constructed via incomplete integrals of the hypergeometric function $\,_2F_1 (\tfrac{1}{2}, \tfrac{1}{2}; \tfrac{1}{2}; \bullet)$; Shen showed that incomplete integrals of $\,_2F_1 (\tfrac{1}{3}, \tfrac{2}{3}; \tfrac{1}{2}; \bullet)$ give rise to an elliptic function ${\rm dn}_3$ that is a signature-three version of the classical Jacobian function  ${\rm dn}$. 

\medbreak 

In [2016] Shen took a quite different approach, analyzing the solutions to certain differential equations that involve the Chebyshev polynomials: explicitly, he studied $y'^2 = T_n(y) - (1 - 2 \mu^2)$ with modulus $0 < \mu < 1$. In general, the solutions to these equations are hyperelliptic: in signatures six ($n = 3$) and four ($n = 4$) these solutions are actually elliptic; in signature three ($n = 6$) the solution $y_6$ is not elliptic, but its square is elliptic. 

\medbreak 

The elliptic functions ${\rm dn}_3$ of [2004] and $y_6^2$ of [2016] both have order two; of course, each has its own coperiodic Weierstrass $\wp$ function. These two $\wp$ functions are not the same: here we determine the precise relationship between them; this relationship naturally engenders a relationship between ${\rm dn}_3$ and $ y_6^2$. Our account includes sufficient detail to make it essentially self-contained. 

\medbreak 

In a separate paper we shall address signature four, connecting the elliptic function ${\rm dn}_2$ that Shen constructed in [2014] to the elliptic function $y_4$ that he presented in [2016]. 

\medbreak 

\section*{Signature three elliptic functions}

\medbreak 

Fix a modulus $\kappa$ in the range $0 < \kappa < 1$ and let $\lambda = \sqrt{1 - \kappa^2} \in (0, 1)$ be the complementary modulus. For convenience, we may write 
$$F = \,_2F_1 (\tfrac{1}{3}, \tfrac{2}{3}; \tfrac{1}{2}; \bullet)$$
as an abbreviation for the indicated hypergeometric function. For the standard evaluation 
$$F(\sin^2 z) = \frac{\cos \frac{1}{3}z}{\cos z}$$
we refer to item (11) on page 101 of [1953] or Entry 35(iii) on page 99 of [1989]. 

\medbreak 

The elliptic function ${\rm dn}_3$ is constructed in [2004] as follows. 

\medbreak 

The assignment 
$$T \mapsto \int_0^T F(\kappa^2 \sin^2 t) \, {\rm d} t$$ 
inverts to define a function $\phi$ with $\phi(0) = 0$ and 
$$u = \int_0^{\phi(u)} F(\kappa^2 \sin^2 t) \, {\rm d} t$$
both as a strictly increasing bijection from $\mathbb{R}$ to $\mathbb{R}$ and as a holomorphic function on a suitably small disc around $0$ in $\mathbb{C}$.  Define $\omega \in \mathbb{R}$ by $\phi(\omega) = \frac{1}{2} \pi$ so that 
$$\omega = \int_0^{\frac{1}{2} \pi} F(\kappa^2 \sin^2 t) \, {\rm d} t.$$
On account of the standard identity  
$$\int_0^{\frac{1}{2} \pi} \,_2F_1 (\tfrac{1}{3}, \tfrac{2}{3}; \tfrac{1}{2}; \kappa^2 \sin^2 t) \, {\rm d} t = \tfrac{1}{2} \pi \,_2F_1 (\tfrac{1}{3}, \tfrac{2}{3}; 1; \kappa^2)$$
it follows that 
$$\omega = \tfrac{1}{2} \pi \,_2F_1 (\tfrac{1}{3}, \tfrac{2}{3}; 1; \kappa^2).$$
An elementary integral calculation shows that if $u \in \mathbb{R}$ then 
$$\phi(u + 2 \omega) = \phi(u) + \pi;$$ 
accordingly, the derivative $\phi ' : \mathbb{R} \to \mathbb{R}$ has period $2 \omega$. This derivative extends analytically to a small enough disc about each point of $\mathbb{R}$ and therefore to a narrow enough band about $\mathbb{R}$ in $\mathbb{C}$. In fact, it is shown in [2004] that $\phi'$ extends to an elliptic function; this is the elliptic function ${\rm dn}_3$ of Shen. 

\medbreak 

In brief, the verification that $\phi'$ extends elliptically proceeds as follows. For convenience, let us write $\delta: = \phi'$. The auxiliary function $\psi: \mathbb{R} \to \mathbb{R}$ is defined by 
$$\psi = \arcsin (\kappa \sin \phi)$$
and also extends holomorphically to a band about $\mathbb{R}$ in $\mathbb{C}$. By differentiation  
$$\delta = \phi' = \frac{1}{F(\kappa^2 \sin^2 \phi)} = \frac{1}{F(\sin^2 \psi)} = \frac{\cos \psi}{\cos \frac{1}{3} \psi}$$
and by trigonometric triplication
$$\cos \psi = 4 \cos^3 \tfrac{1}{3} \psi - 3 \cos \tfrac{1}{3} \psi$$
whence by substitution and rearrangement 
$$\delta^3 + 3 \delta^2 = 4 \cos^2 \psi = 4(1 - \sin^2 \psi) = 4(1 - \kappa^2 \sin^2 \phi).$$
Differentiation here followed by cancellation of a factor $\delta$ on the left and $\phi'$ on the right yields 
$$3 (\delta + 2) \delta' = - 8 \kappa^2 \sin \phi \cos \phi$$
so that 
$$9 (\delta + 2)^2 (\delta')^2 = 64 \kappa^4 \sin^2 \phi \cos^2 \phi.$$ 
As the foregoing cubic in $\delta$ rearranges to give 
$$4 \kappa^2 \sin^2 \phi = (1 - \delta)(2 + \delta)^2 \; \; {\rm and} \; \; 4 \kappa^2 \cos^2 \phi = \delta^3 + 3 \delta^2 - 4 \lambda^2$$ 
we arrive at the following property of $\delta$. 

\begin{theorem} \label{dn3}
The function $\delta = \phi'$ satisfies $\delta(0) = 1$ and the differential equation 
$$9 (\delta')^2 = 4 (1 - \delta) (\delta^3 + 3 \delta^2 - 4 \lambda^2).$$ 
\end{theorem}

\begin{proof} 
Done. See [2004] for the first account. 
\end{proof} 

\medbreak 

We may solve this differential equation explicitly. In fact, the successive substitutions 
$$r = \frac{1}{1 - \delta}, \; q = \tfrac{4}{9} \kappa^2 r, \; p = q - \tfrac{1}{3}$$
result in the differential equation
$$(p')^2 = 4 p^3 - g_2 p - g_3$$
with 
$$g_2 = \tfrac{4}{27} (8 \lambda^2 + 1) \; \; {\rm and} \; \; g_3 = \tfrac{8}{729} (8 \lambda^4 + 20 \lambda^2 - 1)$$
while the initial condition $\delta(0) = 1$ gives $p$ a pole at $0$. This proves that $p$ is the Weierstrass function $\wp (\bullet; g_2, g_3)$ with the indicated invariants. Reversing the substitutions verifies that $\delta = \phi'$ does extend to an elliptic function as claimed. 

\medbreak 

\begin{theorem} \label{p}
The derivative $\phi' : \mathbb{R} \to \mathbb{R}$ extends to the elliptic function ${\rm dn}_3$ given by 
$$(1 - {\rm dn}_3) (\tfrac{1}{3} + p) = \tfrac{4}{9} \kappa^2$$
where $p = \wp(\bullet; g_2, g_3)$ is the Weierstrass function with invariants 
$$g_2 = \tfrac{4}{27} (8 \lambda^2 + 1) \; \; {\rm and} \; \; g_3 = \tfrac{8}{729} (8 \lambda^4 + 20 \lambda^2 - 1).$$
\end{theorem} 

\begin{proof} 
Done. See [2004] for the original argument.
\end{proof} 

\medbreak 

This theorem makes it plain that the elliptic functions ${\rm dn}_3$ and $p$ have exactly the same periods. The invariants $g_2$ and $g_3$ being real and the discriminant 
$$g_2^3 - 27 g_3^2 = (\tfrac{16}{27})^3 \kappa^6 (1 - \kappa^2)$$
being positive, these elliptic functions have a rectangular period lattice. We have already encountered the positive real fundamental half-period: it is precisely 
$$\omega = \tfrac{1}{2} \pi \,_2F_1 (\tfrac{1}{3}, \tfrac{2}{3}; 1; \kappa^2).$$
We shall show later (see Theorem \ref{oo'}) that the purely-imaginary fundamental half-period with positive imaginary part is 
$$\omega' = {\rm i} \tfrac{\sqrt3}{2} \pi \,_2F_1 (\tfrac{1}{3}, \tfrac{2}{3}; 1; \lambda^2).$$

\medbreak 

We now develop some alternative expressions for the real half-period $\omega$. 

\medbreak 

Let us introduce the (acute) modular angle $\alpha \in (0, \tfrac{1}{2} \pi)$ by 
$$\kappa = \sin \alpha$$
so that also
$$\lambda = \cos \alpha.$$ 

\medbreak 

\begin{theorem} \label{evaltrig} 
$$\tfrac{1}{2} \pi \,_2F_1 (\tfrac{1}{3}, \tfrac{2}{3}; 1; \kappa^2) = \sqrt2 \int_0^{\alpha} \frac{\cos \frac{1}{3} \psi}{\sqrt{\cos 2 \psi - \cos 2 \alpha}} {\rm d} \psi.$$
\end{theorem} 

\begin{proof} 
In the expression 
$$\omega = \int_0^{\frac{1}{2} \pi} F(\kappa^2 \sin^2 \phi) \, {\rm d} \phi$$
make the substitution 
$$\sin \psi = \kappa \sin \phi$$ 
so that 
$$\frac{{\rm d} \phi}{{\rm d} \psi} = \frac{\cos \psi}{\kappa \cos \phi}$$
and take into account the hypergeometric evaluation 
$$F(\sin^2 \psi) = \frac{\cos \frac{1}{3} \psi}{\cos \psi}$$
to deduce that 
$$\omega = \int_0^{\alpha} F(\sin^2 \psi) \frac{{\rm d} \phi}{{\rm d} \psi} {\rm d} \psi = \int_0^{\alpha}\frac{\cos \frac{1}{3} \psi}{\cos \psi}\frac{\cos \psi}{\kappa \cos \phi} {\rm d} \psi$$
where 
$$\kappa \cos \phi = \sqrt{\sin^2 \alpha - \sin^2 \psi} = \sqrt{\tfrac{1}{2}(\cos 2 \psi - \cos 2 \alpha)}.$$
\end{proof} 

\medbreak 

We remark that the two sides of this identity are bridged by the Legendre function $P_{-1/6}$, which is related to the left by its standard (Murphy) hypergeometric representation and to the right by its Mehler-Dirichlet integral representation. In fact, the general form of this identity (with $P_{\nu}$ as bridge)  is invoked at equation (3.2) in [2016] and is crucial to the development therein. 

\medbreak 

We now gently massage this formula for $\omega$ to produce another formula for $\omega$ that makes direct contact with [2016]. Here, $T_6$ is the degree six Chebyshev polynomial (`of the first kind'): thus, $T_6(\cos t) = \cos 6 t.$  

\medbreak 

\begin{theorem} \label{evalchebyshev}
$$\tfrac{1}{2} \pi \,_2F_1 (\tfrac{1}{3}, \tfrac{2}{3}; 1; \kappa^2) = \sqrt6 \int_{\cos \tfrac{1}{3} (\pi + \alpha)}^{\cos \tfrac{1}{3} (\pi - \alpha)} \frac{1}{\sqrt{T_6 (x) - \cos 2 \alpha}} \, {\rm d} x\, .$$
\end{theorem} 

\begin{proof} 
We first make use of the addition formula 
$$\sin \tfrac{1}{3} (\pi + \psi) = \sin \tfrac{1}{3} \pi \cos \tfrac{1}{3} \psi + \cos \tfrac{1}{3} \pi \sin \tfrac{1}{3} \psi$$
along with the evenness of cosine and the oddness of sine, to deduce from Theorem \ref{evaltrig} that 
$$\omega = \frac{1}{\sqrt2} \int_{ - \alpha}^{\alpha} \frac{\cos \frac{1}{3} \psi}{\sqrt{\cos 2 \psi - \cos 2 \alpha}} {\rm d} \psi = \sqrt{\frac{2}{3}} \int_{- \alpha}^{\alpha} \frac{\sin \frac{1}{3} (\pi + \psi)}{\sqrt{\cos 2 \psi - \cos 2 \alpha}} {\rm d} \psi;$$
we next substitute $\theta = \pi + \psi$ to deduce that 
$$\omega = \sqrt{\frac{2}{3}} \int_{\pi - \alpha}^{\pi + \alpha} \frac{\sin \frac{1}{3} \theta}{\sqrt{\cos 2 \theta - \cos 2 \alpha}} {\rm d} \theta$$
and finally substitute $x = \cos \frac {1}{3} \theta$ to conclude that 
$$\omega = \sqrt6 \int_{\cos \tfrac{1}{3} (\pi + \alpha)}^{\cos \tfrac{1}{3} (\pi - \alpha)} \frac{1}{\sqrt{T_6 (x) - \cos 2 \alpha}} \, {\rm d} x$$
as claimed. 
\end{proof} 

\medbreak 

This formula for $\omega$ places us firmly in the context of [2016] and serves as a convenient point of entry to the signature-three portion of that paper: it prompts us to consider the differential equation 
$$(w')^2 = T_6 (w) - (1 - 2 \kappa^2)$$ 
in which we have reinstated $\kappa = \sin \alpha$ so that $\cos 2 \alpha = 1 - 2 \kappa^2$; we solve this differential equation subject to the initial condition $w(0) = 0$. To effect a solution, let us write 
$$W = T_2(w) = 2 w^2 - 1$$
so that 
$$T_6(w) = T_3 (T_2(w)) = 4 W^3 - 3 W.$$
Now the foregoing differential equation for $w$ becomes the following differential equation for $W$. 
\medbreak 

\begin{theorem} \label{W}
The function $W$ satisfies $W(0) = -1$ and the differential equation 
$$(W')^2 = 8(W + 1) (4W^3 - 3W - (1 - 2 \kappa^2)).$$ 
\end{theorem} 

\begin{proof} 
Simple substitution. See [2016] for an alternative route here. 
\end{proof} 

\medbreak 

We may solve this initial value problem explicitly, as follows. 

\medbreak 

\begin{theorem} \label{P}
$W$ is the elliptic function given by 
$$(1 + W) (6 - P) = 4 \lambda^2.$$ 
where $P = \wp(\bullet; G_2, G_3)$ is the Weierstrass function with invariants 
$$G_2 = 48(1 + 8 \kappa^2) \; \; {\rm and} \; \; G_3 = 64 (1 - 20 \kappa^2 - 8 \kappa^4).$$ 
\end{theorem} 

\begin{proof} 
The successive substitutions 
$$R = \frac{1}{W + 1}, \; Q = - 4 \lambda^2 R, \; P = Q + 6$$ 
produce a function $P$ that has a pole at the origin and satisfies the differential equation 
$$(P')^2 = 4 P^3 - G_2 P - G_3.$$ 
It follows that $P$ is the Weierstrass function with the indicated invariants and that $W$ is as claimed. 
\end{proof} 

\medbreak 

This theorem makes it plain that the elliptic functions $W$ and $P$ have exactly the same periods. The invariants being real and the discriminant positive, these elliptic functions have a rectangular period lattice; let $(2 \, \Omega, 2 \, \Omega')$ be the fundamental pair of periods in which the real $\Omega$ is positive and the purely imaginary $\Omega'$ has positive imaginary part. We shall derive explicit hypergeometric expressions for $\Omega$ and $\Omega'$ in Theorem \ref{O} and Theorem \ref{O'}. 

\medbreak 

The Weierstrass $\wp$ functions $p$ and $P$ are of course connected by the fact of our passage from [2004] to [2016]; the connexion between them is brought out explicitly by a glance at their invariants, as displayed in Theorem \ref{p} and Theorem \ref{P}. This precise connexion involves a twist, to highlight which we decorate these $\wp$ functions by the modulus that enters the construction of ${\rm dn}_3$: thus, $p_{\kappa}$ is the $\wp$ function that appears in Theorem \ref{p} while the $\wp$ function $p_{\lambda}$ has invariants $g_2 =  \tfrac{4}{27} (8 \kappa^2 + 1)$ and $g_3 = \tfrac{8}{729} (8 \kappa^4 + 20 \kappa^2 - 1)$; the $\wp$ function $P_{\kappa}$ has invariants  $G_2 = 48(1 + 8 \kappa^2)$ and $G_3 = 64 (1 - 20 \kappa^2 - 8 \kappa^4)$ as in Theorem \ref{P}. 

\medbreak 

\begin{theorem} \label{pP}
The Weierstrass functions $p$ and $P$ are related by the identity 
$$P_{\kappa} (z) = - 18 p_{\lambda} (3 \sqrt2 \, {\rm i} z).$$ 
\end{theorem} 

\begin{proof} 
We invoke the familiar homogeneity relation for $\wp$ functions in the form 
$$\wp(z; \gamma^4 g_2, \gamma^6 g_3) = \gamma^{2} \wp(\gamma z: g_2, g_3).$$
Here, either of the choices
$$\gamma = \pm 3 \sqrt2 \, {\rm i}$$
ensures that the invariants specified immediately prior to the theorem satisfy 
$$G_2 = \gamma^4 g_2 \; \; {\rm and} \; \; G_3 = \gamma^6 g_3$$
while the multiplier $\gamma^2$ equals $- 18$. The sign ambiguity in $\gamma$ is immaterial, as $p_{\lambda}$ is even. 
\end{proof} 

\medbreak 

This relationship between $P_{\kappa}$ and $p_{\lambda}$ implies that their period lattices are related by a scaling and a right-angle rotation: explicitly, the fundamental half-periods $\omega_{\lambda}, \, \omega_{\lambda}'$ of $p_{\lambda}$ and $\Omega_{\kappa}, \, \Omega_{\kappa}'$ of $P_{\kappa}$ are related by 
$$\omega_{\lambda} = - 3 \sqrt2 \, {\rm i} \Omega_{\kappa}'$$ 
and 
$$\omega_{\lambda}' = 3 \sqrt2 \, {\rm i} \Omega_{\kappa}.$$ 
More generally, this relationship between $\wp$ functions allows us to transfer knowledge derived in the setting of [2004] to the setting of [2016] and ${\it vice \; versa}.$ 

\medbreak 

We now embark on explicit evaluations of $\Omega_{\kappa}$ and $\Omega_{\kappa}'$ in hypergeometric terms. 

\medbreak 

We calculate the periods $\Omega = \Omega_{\kappa}$ and $\Omega' = \Omega_{\kappa}'$ of $P = P_{\kappa}$ by calculating the (identical) periods of the coperiodic elliptic function $W$ in Theorem \ref{P}. Let $\mathbb{K}$ be the `half-period' rectangle with vertices 
$$0, \, \Omega, \, \Omega + \Omega', \, \Omega', \, 0.$$ 
It is a familiar fact that as the perimeter of the rectangle $\mathbb{K}$ is traversed in this counterclockwise direction, the values of the $\wp$ function $P$ are real and decrease strictly from $+ \infty$ to $- \infty$. Moreover, the values of $P$ at the vertices $\Omega, \, \Omega + \Omega', \, \Omega$ are critical; the corresponding values of $W$ are the zeros of the cubic $4 W^3 + 3 W - \cos 2 \alpha$,  namely  
$$W(\Omega) = \cos \tfrac{2}{3} \alpha, \, W(\Omega + \Omega') = \cos \tfrac{2}{3} (\pi - \alpha), \, W(\Omega') = \cos \tfrac{2}{3} (\pi + \alpha)$$
in decreasing order, with $W(\Omega) > 0 > W(\Omega')$ and $W(\Omega + \Omega')$ between the extremes, its sign dependent on $\kappa$. The function $W$ has a pole (where $P = 6$) at a point on the lower edge $(0, \Omega)$ of $\mathbb{K}$ whose precise location will be revealed after Theorem \ref{oo'}; the values of $W$ elsewhere on the perimeter of $\mathbb{K}$ are of course real. 

\medbreak 

Now, we calculate $\Omega = \Omega_{\kappa}$ by integration along the upper edge of $\mathbb{K}$ (as integration along the lower edge would encounter the pole). 

\medbreak 

\begin{theorem} \label{O}
$$\Omega_{\kappa} = \tfrac{1}{2 \sqrt6} \pi \, _2F_1 (\tfrac{1}{3}, \tfrac{2}{3}; 1; \kappa^2).$$
\end{theorem} 

\begin{proof} 
Parametrize the top edge of $\mathbb{K}$ by $z = t + \Omega'$ for $0 \leqslant t \leqslant \Omega$. It may be checked that $W$ increases with $t$ here, so 
$$\frac{{\rm d} W}{{\rm d} t} = \sqrt{8 (W + 1)(4 W^3 + 3 W - \cos 2 \alpha)}$$
with the positive square-root, whence 
$$\Omega = \int_0^{\Omega} 1 = \int_{\cos \tfrac{2}{3} (\pi + \alpha)}^{\cos \tfrac{2}{3} (\pi - \alpha)} \frac{1}{\sqrt{8 (W + 1)(4 W^3 + 3 W - \cos 2 \alpha)}} \, {\rm d} W.$$
In the last integral, substitute $W = 2 w^2 - 1$ to deduce that 
$$\Omega =  \int_{\cos \tfrac{1}{3} (\pi + \alpha)}^{\cos \tfrac{1}{3} (\pi - \alpha)} \frac{1}{\sqrt{T_6(w) - \cos 2 \alpha}} \, {\rm d} w$$
and summon Theorem \ref{evalchebyshev} to conclude the proof. 
\end{proof} 

\medbreak 

Next, we calculate $\Omega' = \Omega_{\kappa}'$ by integration up the right edge of $\mathbb{K}$. 

\medbreak 

\begin{theorem} \label{O'}
$$\Omega_{\kappa}' = {\rm i} \tfrac{1}{6 \sqrt2} \pi \, _2F_1 (\tfrac{1}{3}, \tfrac{2}{3}; 1; \lambda^2).$$
\end{theorem} 

\begin{proof} 
Put $x(t) = W(\Omega + {\rm i} t)$ for $0 \leqslant t \leqslant \Omega'$. The differential equation for $W$ corresponds to the differential equation 
$$(x')^2 = 8 (x + 1) (\cos 2 \alpha + 3 x - 4 x^3)$$
for $x$ and $W$ decreases up the right edge, so 
$$\frac{{\rm d} x}{{\rm d} t} = - \sqrt{8 (x + 1)(\cos 2 \alpha + 3 x - 4 x^3)}.$$
Integrate to get  
$$\Omega' = \int_0^{\Omega'} 1 \, {\rm d} t = \int_{\cos \frac{2}{3} (\pi - \alpha)}^{\cos \frac{2}{3} \alpha} \frac{1}{\sqrt{8 (x + 1)(\cos 2 \alpha +3 x - 4 x^3)}} \, {\rm d} x$$
and substitute $x = \cos \frac{2}{3} \theta$ to get 
$$\Omega' = \frac{1}{3} \int_{\alpha}^{\pi - \alpha} \frac{\sin \frac{1}{3} \theta}{\sqrt{\cos 2 \alpha - \cos 2 \theta}} \, {\rm d} \theta$$
after trigonometric duplication in numerator and denominator. The substitution $\theta = \frac{1}{2} - \psi$ now yields 
$$\Omega' = \frac{1}{3} \int_0^{\frac{1}{2} \pi - \alpha}  \frac{\cos \frac{1}{3} \psi}{\sqrt{\cos 2 \alpha  + \cos 2 \psi}} \, {\rm d} \psi$$
when trigonometric addition and parity are taken into account as in the proof of Theorem \ref{evalchebyshev}. Let $\beta = \frac{1}{2} - \alpha$: then $\sin \beta = \lambda$ and 
$$\Omega' = \frac{1}{3} \int_0^{\beta} \frac{\cos \frac{1}{3} \psi}{\sqrt{\cos 2 \psi  - \cos 2 \beta}} \, {\rm d} \psi$$
and Theorem \ref{evaltrig} concludes the proof. 
\end{proof} 

\medbreak 

Of course, we can identify the half-periods of $\delta = {\rm dn}_3$ along the lines of Theorem \ref{O} and Theorem \ref{O'}; we outline the identification of the real half-period $\omega = \omega_{\kappa}$. By virtue of the differential equation in Theorem \ref{dn3} and the fact that $\delta$ decreases along the lower edge $(0, \omega)$ of the `half-period' rectangle, from its value $1$ at $0$ to its value $2 \cos \frac{2}{3} \alpha - 1$ at $\omega$, we deduce that 
$$\omega = - \frac{3}{2} \int_1^{2 \cos \frac{2}{3} \alpha - 1} \frac{1}{\sqrt{(1 - \delta) (\delta^3 + 3 \delta^2 - 4 \lambda^2)}} \, {\rm d} \delta$$ 
whence the substitution $\delta = 2 \cos \frac{2}{3} t - 1 = 4 \cos^2 \frac{1}{3} t -3 = 1 - 4 \sin^2 \frac{1}{3} t$ yields 
$$\omega = \sqrt2 \int_0^{\alpha} \frac{\cos \frac{1}{3} t}{\sqrt{T_6 (\cos \frac{1}{3}t) - (2 \lambda^2 - 1)}} \, {\rm d} t = \sqrt2 \int_0^{\alpha} \frac{\cos \frac{1}{3} t}{\sqrt{\cos 2 t - \cos 2 \alpha}} \, {\rm d} t$$
and Theorem \ref{evaltrig} completes the identification.  

\medbreak 

Alternatively, we may activate the transfer to which we alluded after Theorem \ref{pP}. 

\medbreak 

\begin{theorem} \label{oo'}
The fundamental half-periods of $p = p_{\kappa}$ are given by 
$$\omega_{\kappa} = \tfrac{1}{2} \pi \, _2F_1 (\tfrac{1}{3}, \tfrac{2}{3}; 1; \kappa^2)$$
and 
$$\omega_{\kappa}' = {\rm i} \, \tfrac{\sqrt3}{2} \pi \, _2F_1 (\tfrac{1}{3}, \tfrac{2}{3}; 1; \lambda^2)$$
\end{theorem} 

\begin{proof} 
Apply the identities following Theorem \ref{pP} to Theorem \ref{O} and Theorem \ref{O'}, switching to the complementary modulus. Naturally, the expression for $\omega_{\kappa}$ agrees with the earlier expression noted after Theorem \ref{p}. 
\end{proof} 

\medbreak 

Our identifications of the fundamental periods have the following consequences for the shapes of the period lattices: on the one hand 
$$\frac{\Omega_{\kappa}'}{\Omega_{\kappa}} = {\rm i} \, \frac{1}{\sqrt3} \; \frac{_2F_1 (\tfrac{1}{3}, \tfrac{2}{3}; 1; 1 - \kappa^2)}{_2F_1 (\tfrac{1}{3}, \tfrac{2}{3}; 1; \kappa^2)};$$ 
on the other hand 
$$\frac{\omega_{\kappa}'}{\omega_{\kappa}} = {\rm i} \, \sqrt3 \; \frac{_2F_1 (\tfrac{1}{3}, \tfrac{2}{3}; 1; 1 - \kappa^2)}{_2F_1 (\tfrac{1}{3}, \tfrac{2}{3}; 1; \kappa^2)}.$$ 
Here, we have written $1 - \kappa^2$ in place of $\lambda^2$ so as to conform with the way in which such ratios are traditionally expressed. 

\medbreak 

Naturally, the transfer noted after Theorem \ref{pP} also applies to the problem of locating poles. From Theorem \ref{p} we see that $\delta$ has poles where $p$ has value $-1/3$; it is proved in [2004] that these poles lie at the points congruent to $\pm \tfrac{2}{3} \omega_{\kappa}'$.  From Theorem \ref{P} we see that $W$ has poles where $P$ has value $6$; it is proved in [2016] that these poles lie at the points congruent to $\pm \tfrac{2}{3} \Omega_{\kappa}$. Theorem \ref{pP} links these determinations (and indeed allows us to deduce the one from the other); in particular, the `${\rm i}$' twist in that theorem explains why $\delta$ has a pair of imaginary poles whereas $W$ has a pair of real poles. 

\medbreak 

We make two further points in this regard. 

\medbreak 

First, notice that these poles are situated at certain of the points of period three, each relative to the respective period lattice; this is how the poles of ${\rm dn}_3$ were located in [2019] and [2020]. In [2004] their placement was established very differently, using theta functions; in [2016] they were found by integrating up the left edge of the half-period rectangle, with the aid of a hyperbolic counterpart to the Mehler-Dirichlet integral formula. 

\medbreak 

Second, recall $W = 2 w^2 - 1$ from the discussion leading to Theorem \ref{W}. The poles of $W$ are clearly simple: indirectly, because the second-order elliptic function $W$ has two poles in its period rectangle; directly, because $P' \neq 0$ where $P = 6$. It follows at once that $W$ lacks meromorphic square-roots; in particular, $w$ cannot be elliptic. We remark that our $w$ is called $y_6$ in [2016]; so $y_6$ is not an elliptic function, although its square is.  

\medbreak

\bigbreak

\begin{center} 
{\small R}{\footnotesize EFERENCES}
\end{center} 
\medbreak 

[1953] A. Erdelyi (director), {\it Higher Transcendental Functions}, Volume 1, McGraw-Hill. 

\medbreak 

[1989] B.C. Berndt, {\it Ramanujan's Notebooks Part II}, Springer-Verlag. 

\medbreak 

[2004] Li-Chien Shen, {\it On the theory of elliptic functions based on $_2F_1(\frac{1}{3}, \frac{2}{3} ; \frac{1}{2} ; z)$}, Transactions of the American Mathematical Society {\bf 357} 2043-2058. 

\medbreak 

[2014] Li-Chien Shen, {\it On a theory of elliptic functions based on the incomplete integral of the hypergeometric function $_2 F_1 (\frac{1}{4}, \frac{3}{4} ; \frac{1}{2} ; z)$}, Ramanujan Journal {\bf 34} 209-225. 

\medbreak 

[2016] Li-Chien Shen, {\it On Three Differential Equations Associated with Chebyshev Polynomials of Degrees 3, 4 and 6}, Acta Mathematica Sinica, English Series {\bf 33} (1) 21-36. 

\medbreak 

[2019] P.L. Robinson, {\it Elliptic functions from $F(\frac{1}{3}, \frac{2}{3} ; \frac{1}{2} ; \bullet)$}, arXiv 1907.09938. 

\medbreak 

[2020] P.L. Robinson, {\it The elliptic function ${\rm dn}_3$ of Shen}, arXiv 2008.13572.

\end{document}